%% file: walk-dim_2016.tex
\documentclass[11pt,reqno]{amsart}

\usepackage{amssymb}
\usepackage{txfonts}
\usepackage{amsfonts}
\usepackage{amsmath}
\usepackage{graphicx}
\usepackage{subfigure}
\usepackage{hyperref}
\usepackage{float}
\usepackage{epstopdf}

\allowdisplaybreaks
\newtheorem{theorem}{Theorem}[section]

\newtheorem{lemma}[theorem]{Lemma}

\newtheorem{example}[theorem]{Example}

\renewcommand{\theequation}{\thesection.\arabic{equation}}
\newenvironment{notation}{\smallskip{\sc Notation.}\rm}{\smallskip}

\input{tcilatex}
\input{tcirm}

\begin{document}
\title[]{Lipschitz invariance of walk dimension on connected self-similar sets}
\author[Gu]{Qingsong Gu}
\address{Department of Mathematics, The Chinese University of Hong Kong, Shatin, N.T., Hong Kong}
\email{001gqs@163.com}
\author[Rao]{Hui Rao}
\address{Department of Mathematics and Statistics, Central China Normal University, Wuhan 430079, China}
\email{hrao@mail.ccnu.edu.cn}
\date{\today}

\begin{abstract}
Walk dimension is an important conception in analysis of fractals.
In this paper we prove that the walk dimension of a connected compact set
 possessing an Alfors regular measure is an invariant under Lipschitz transforms.
 As an application, we show some generalized Sierpi\'nski gaskets are not Lipschitz equivalent.
\end{abstract}

\subjclass[msc2010]{Primary: 35K08 Secondary: {28A80, 35J08, 46E35, 47D07}}
\keywords{walk dimension, Lipschitz invariant, Besov space, heat kernel}
\thanks{The second author is supported by NSFC Nos. 11431007 and 11471075.}
\maketitle
\tableofcontents

%\tableofcontents

\section{\textbf{Introduction}}
Let $(X, d_1)$ and $(Y,d_2)$ be two metric spaces.
We say that $T:(X,d_1)\to (Y,d_2)$ is a bi-Lipschitz transform, if $T$ is a bijection, and there exists a constant $C>0$ such that for any $x,y\in X$,
\begin{equation}
C^{-1}d_1(x,y)\leq d_2(Tx,Ty)\leq Cd_1(x,y).
\end{equation}
The studies of Lipschitz equivalence of self-similar sets are initialled by Falconer and Marsh \cite{FaMa92} and
David and Semmes \cite{DS}.  Rao, Ruan and Xi \cite{RRX06} (2006) answered a question posed by David and Semmes \cite{DS}, by showing that the self-similar sets illustrated in Figure \ref{One-3-5} are Lipschitz equivalence. After that, there are many works devoted to this topic, for example, Xi and Xiong \cite{XiXi10, XiXi12}, Luo and Lau  \cite{LuoLau13},
Ruan, Wang and Xi \cite{RWX12},  and Rao and Zhang \cite{RZ15}.
However, the studies mentioned above are all on self-similar sets which are totally disconnected.
\vskip -0.2cm
\begin{figure}[h]
  \includegraphics[width=0.8\textwidth]{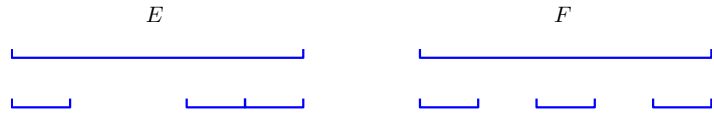}\\
  \vskip -1 cm
  \caption{The Cantor sets $E$ and $F$ are Lipschitz equivalence \cite{RRX06}.}\label{One-3-5}
\end{figure}

 Recently, there are some studies on a class of self-similar sets which are not totally connected.
 A non-empty compact set satisfying the set equation
 $$
 F=\bigcup_{d\in {\mathcal D}} \frac{F+d}{n}
 $$
 is called a \emph{fractal square} if $n\geq 2$ and ${\mathcal D}\subset \{0,1,\dots, n-1\}^2$.
 Rao and Zhu \cite{RaoZhu} studied fractal squares containing line segments.
 Using a certain finite state automaton, they construct a bi-Lipschitz mapping between  the fractals
 illustrated in Figure \ref{fig:rao-zhu}. Ruan and Wang \cite{RuanWang} studied  fractal squares of ratio $1/3$ and with $7$ or $8$ branches,
  which are all connected fractals. They proved that two such sets are Lipschitz equivalent if and only if they are isometric. See Figure \ref{fig:ruan-wang-1}.
  Their method is to show that two such sets are not homeomorphic by various connectivity property, which \textbf{depends} on very careful observations.

\begin{figure}[htbp]
\subfigure[]{
  \includegraphics[width=0.3\textwidth]{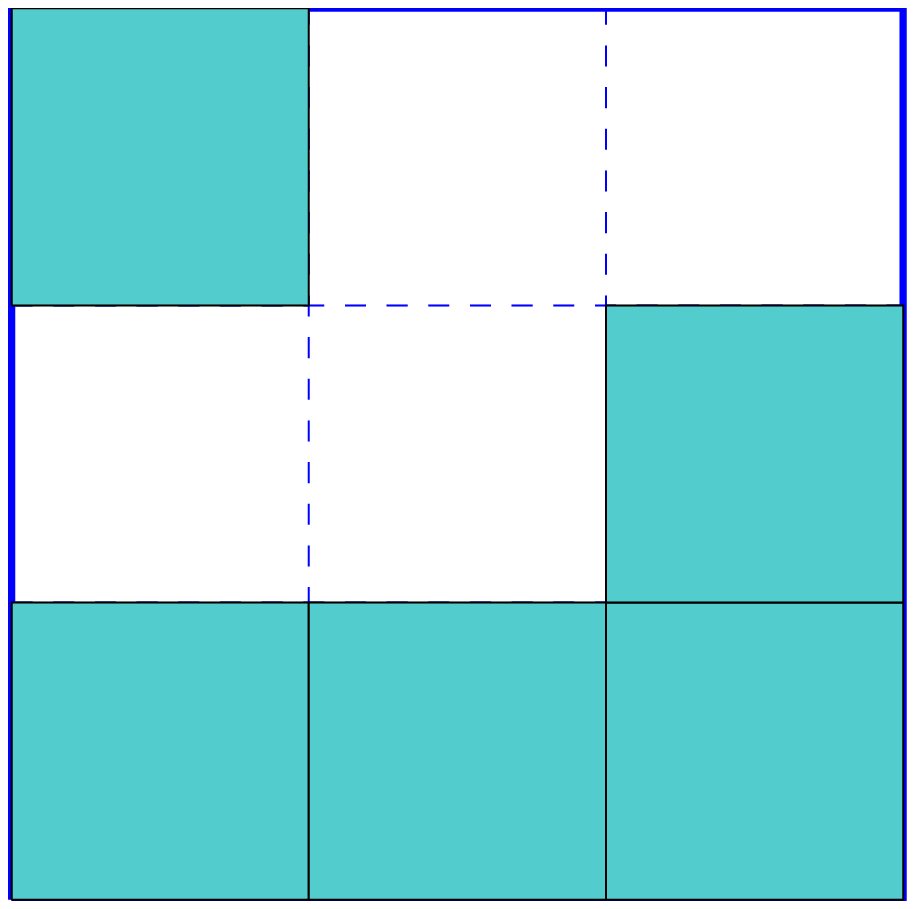}}
  \subfigure[]{
  \includegraphics[width=0.3\textwidth]{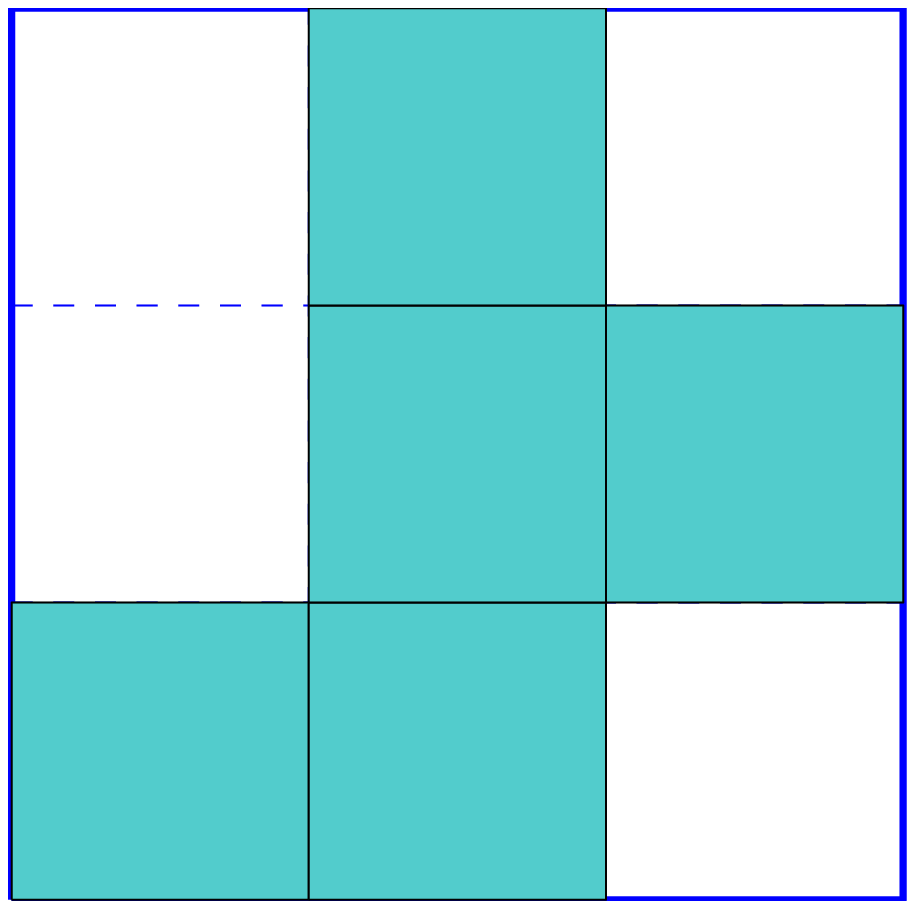}
  }\\
  \subfigure[]{
  \includegraphics[width=0.3\textwidth]{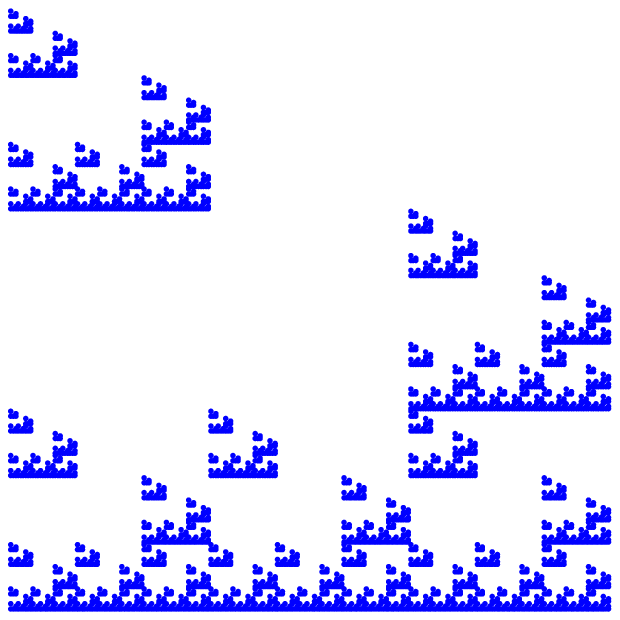}
  }
   \subfigure[]{
  \includegraphics[width=0.3\textwidth]{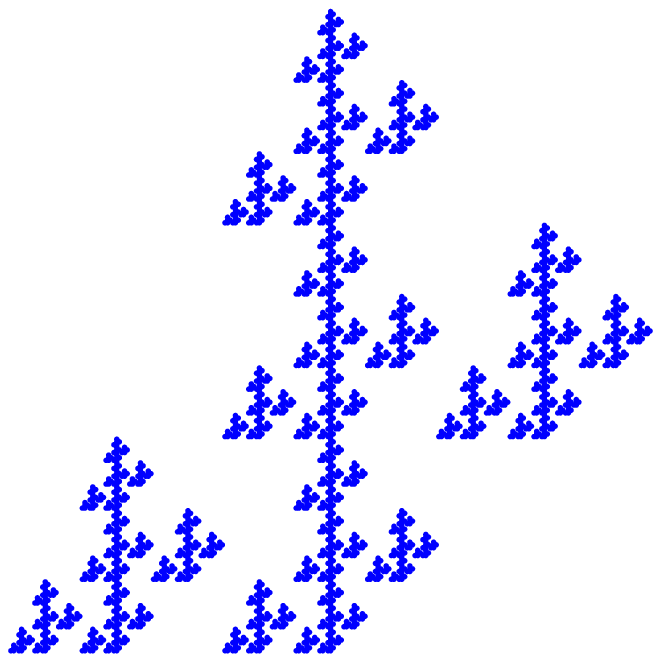}
  }\\
\caption{Two fractal squares containing line segments.}
\label{fig:rao-zhu}
\end{figure}

\begin{figure}[htbp]
\subfigure[]{
  \includegraphics[width=0.3\textwidth]{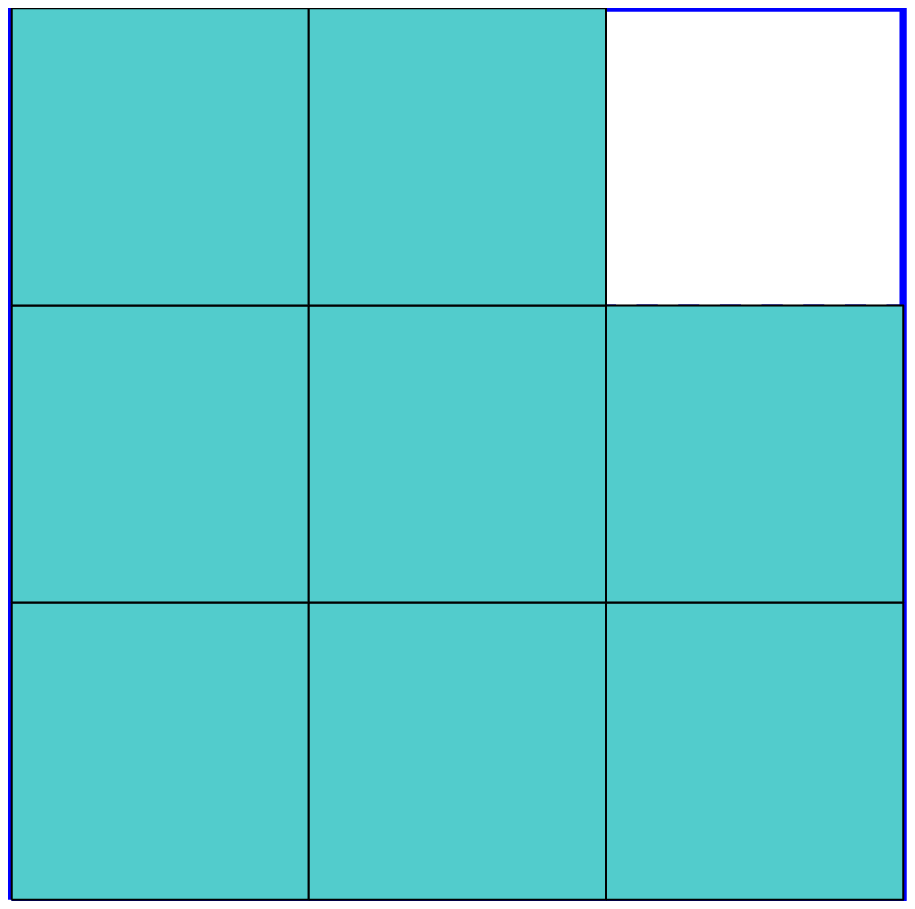}}
  \subfigure[]{
  \includegraphics[width=0.3\textwidth]{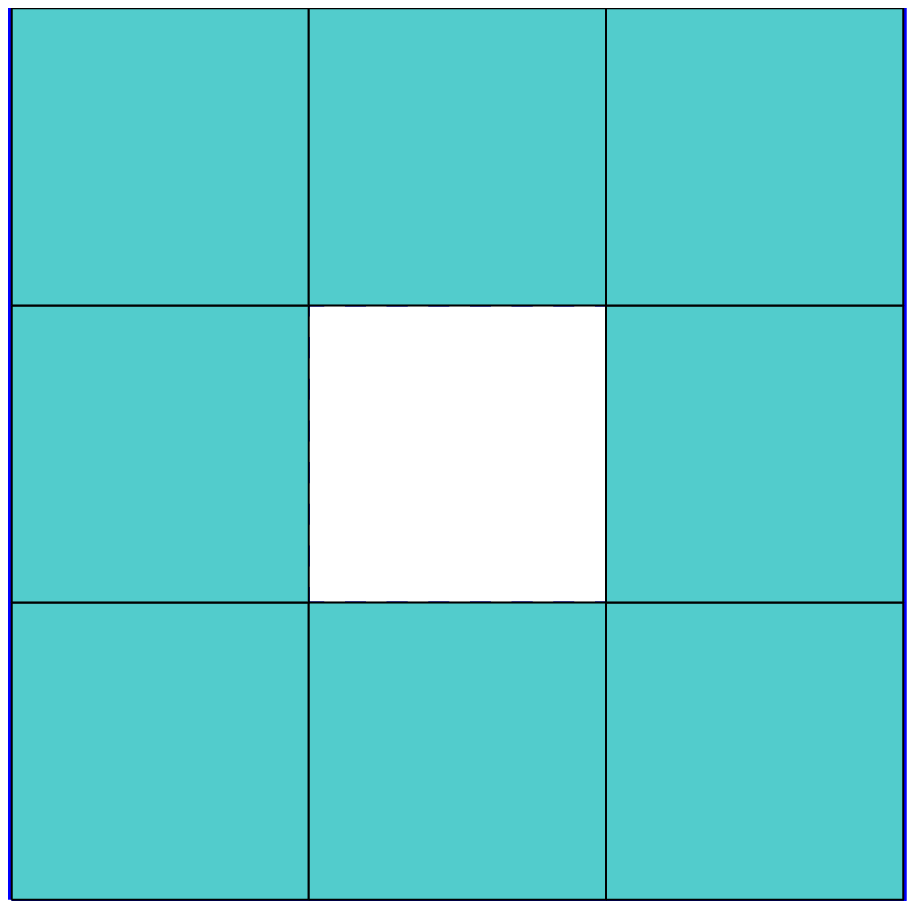}
  }
  \subfigure[]{
  \includegraphics[width=0.3\textwidth]{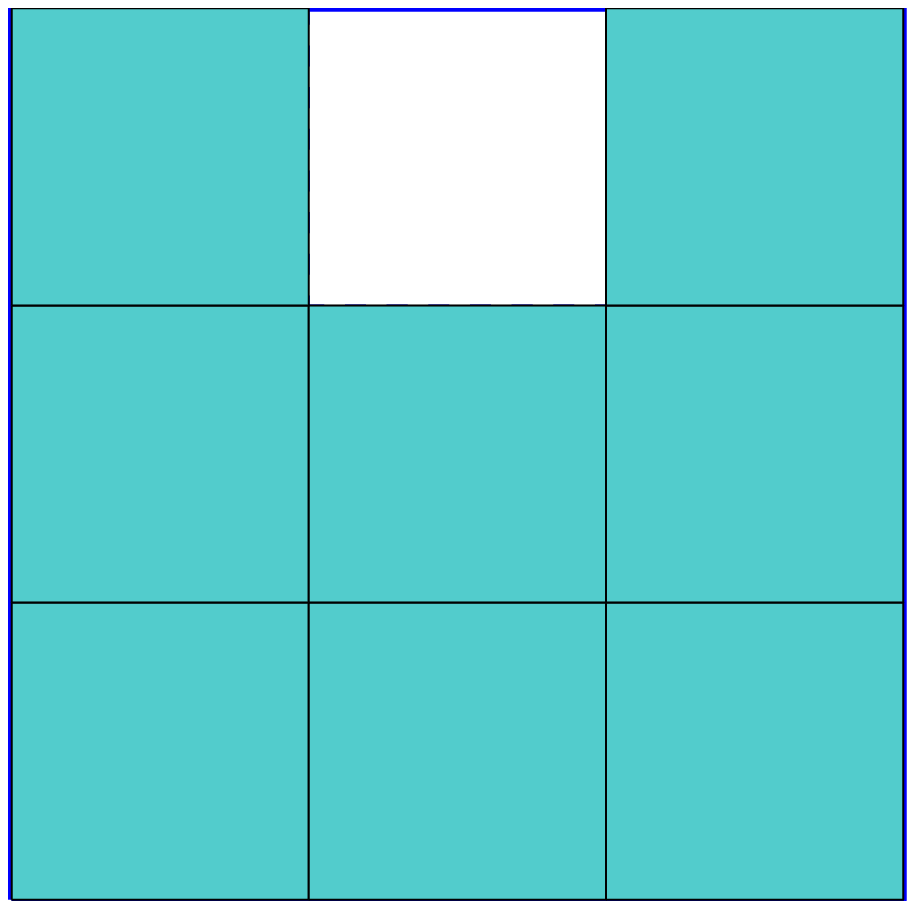}}\\
  \subfigure[]{
  \includegraphics[width=0.3\textwidth]{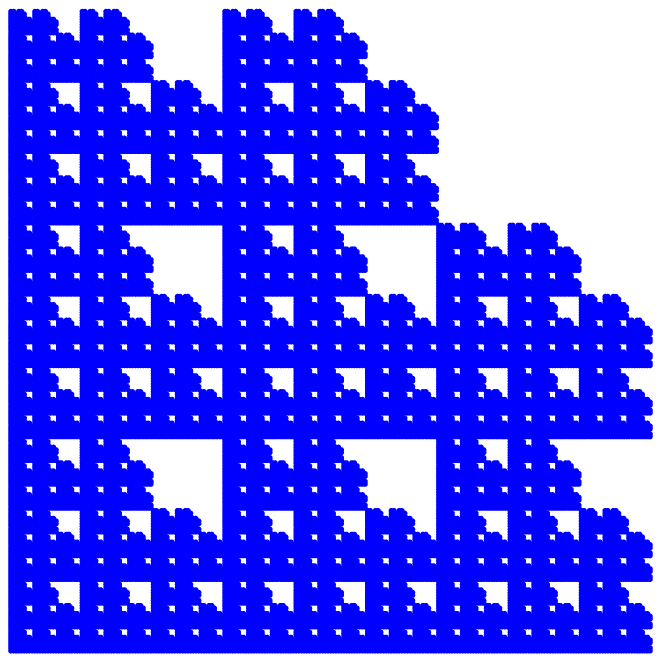}
  }
   \subfigure[]{
  \includegraphics[width=0.3\textwidth]{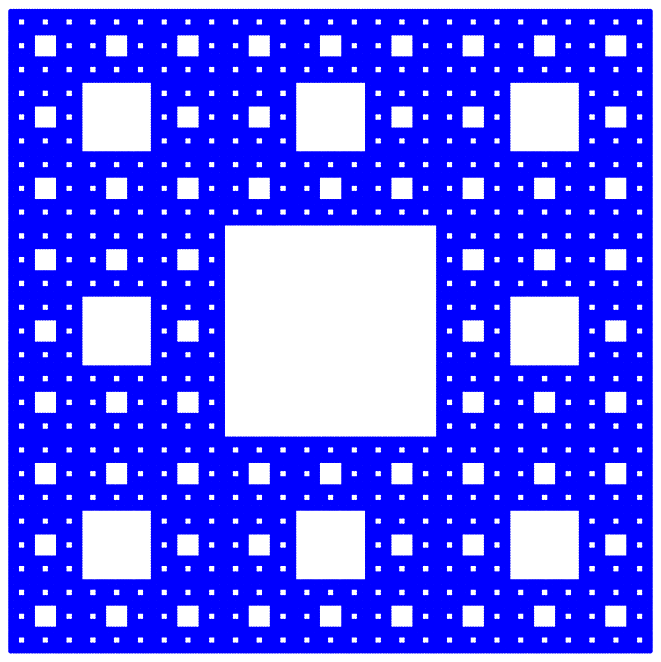}
  }
  \subfigure[]{
  \includegraphics[width=0.3\textwidth]{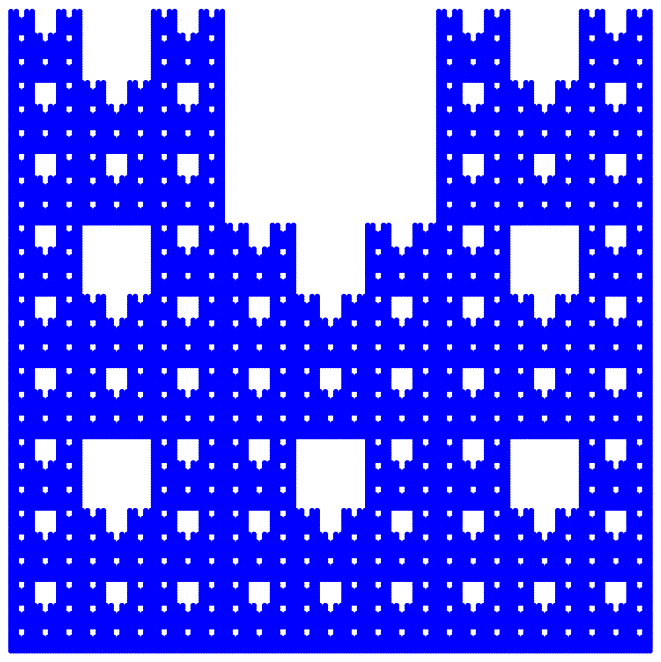}
  }
\caption{Up to isometries, there are three fractal squares with $8$ branches. All of them are not Lipschitz equivalent.}
\label{fig:ruan-wang-1}
\end{figure}

The study of the topological properties of connected self-similar set is a very hard problem.
By Whyburn \cite{Why}, the two fractal squares in Figure \ref{fig:n=5} are homeomorphic.
It is generally believed that these two fractal squares are not Lipschitz equivalent \cite{RRW13}.
To show two sets are not Lipschitz equivalent, the main method is to construct a certain Lipschitz invariant
to distinct them, which is the \textbf{motivation} of the present paper.
%It is an interesting problem that which invariant can distinct the fractals in Figure 3.

In this paper, we use the theory of Laplacian on fractals to construct Lipschitz invariant.
We show that the critical exponent $\beta^\ast$ defined in Grigor'yan, Hu and Lau \cite[Definition 4.4]{GRIHL03} is an Lipschitz invariant.
It is shown that \cite{GRIHL03} this critical value coincides with the \emph{walk dimension} under a mild assumption.
%We then give some examples of self-similar sets as applications of our result.
(See also Theorem 3.1.)
 Due to the difficulty of computing,
 the walk dimension can be obtained for few self-similar sets.  Nevertheless, we hope our study may shed some light to the study of Lipschitz equivalence of connected fractals.

\begin{figure}[htbp]
\subfigure[]{
  \includegraphics[width=0.3\textwidth]{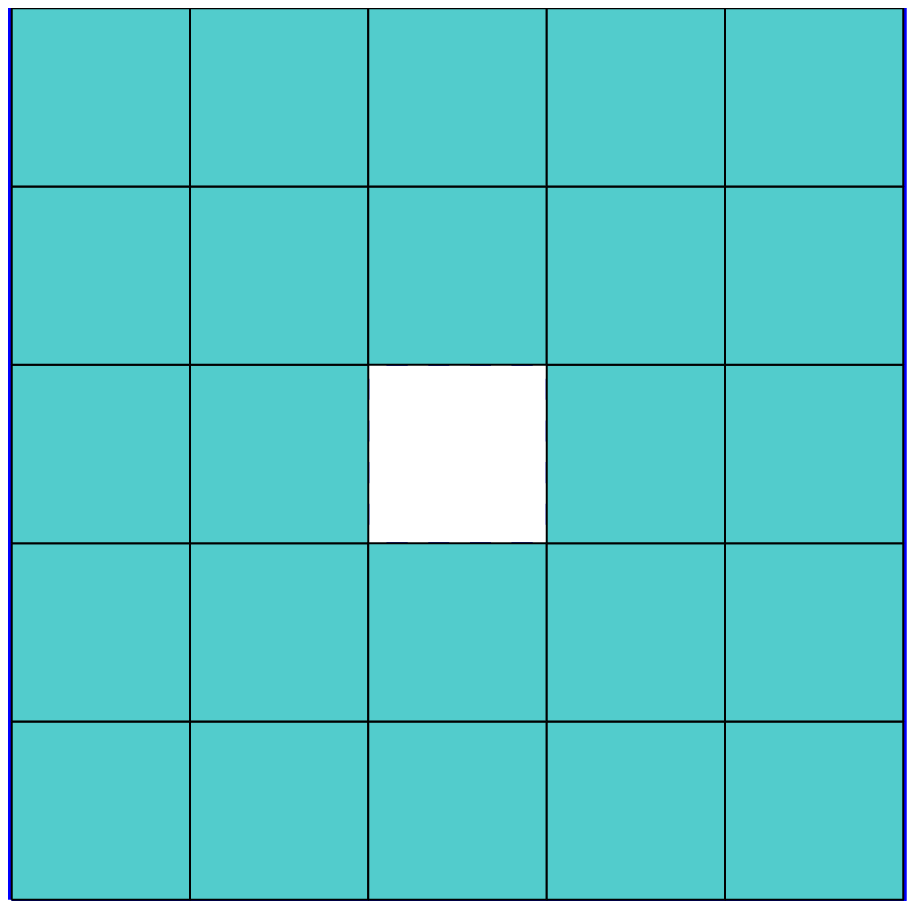}}
  \subfigure[]{
  \includegraphics[width=0.3\textwidth]{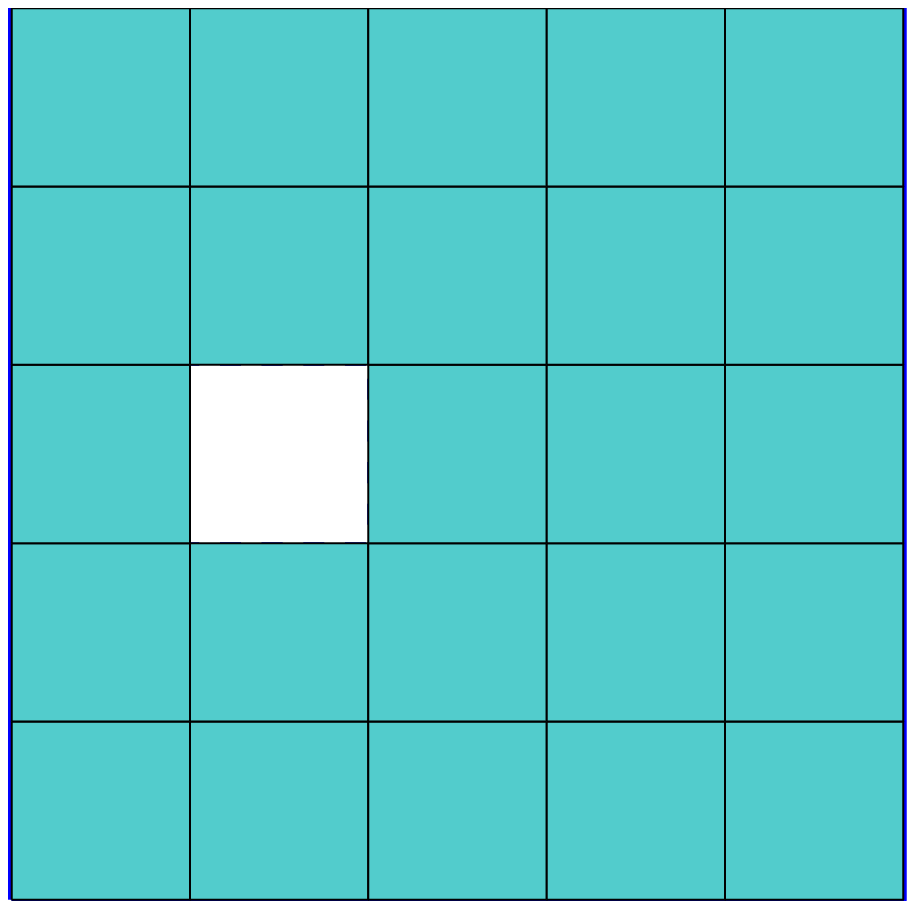}
  }\\
  \subfigure[]{
  \includegraphics[width=0.3\textwidth]{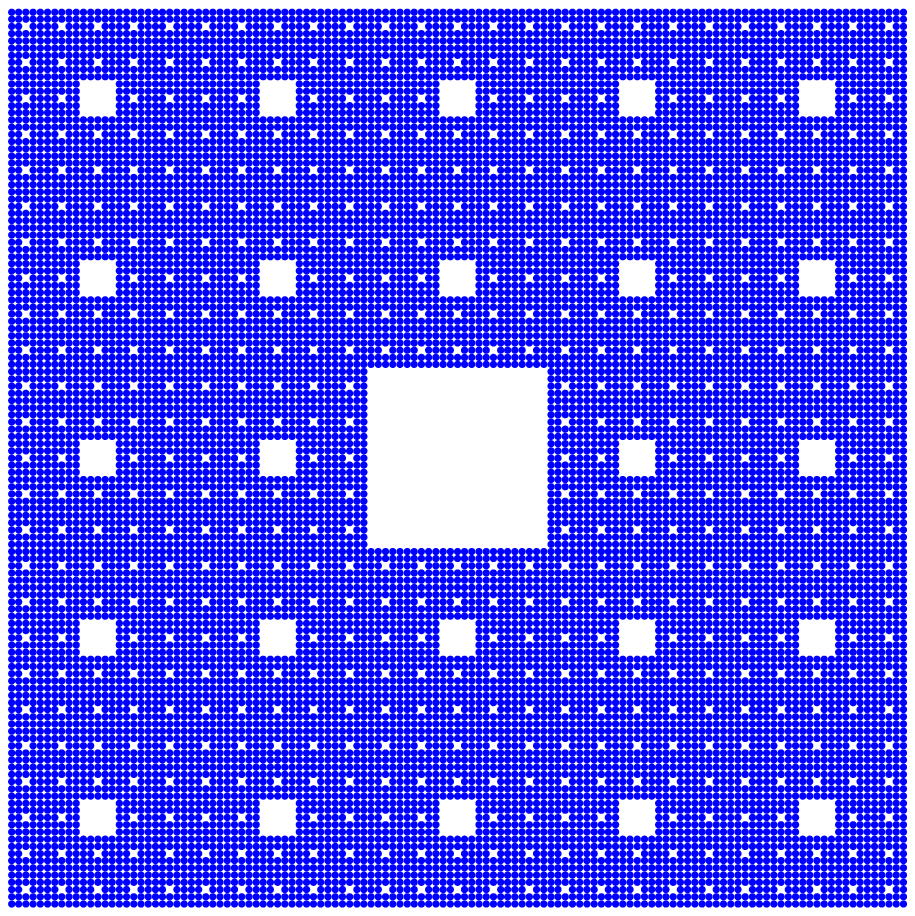}
  }
   \subfigure[]{
  \includegraphics[width=0.3\textwidth]{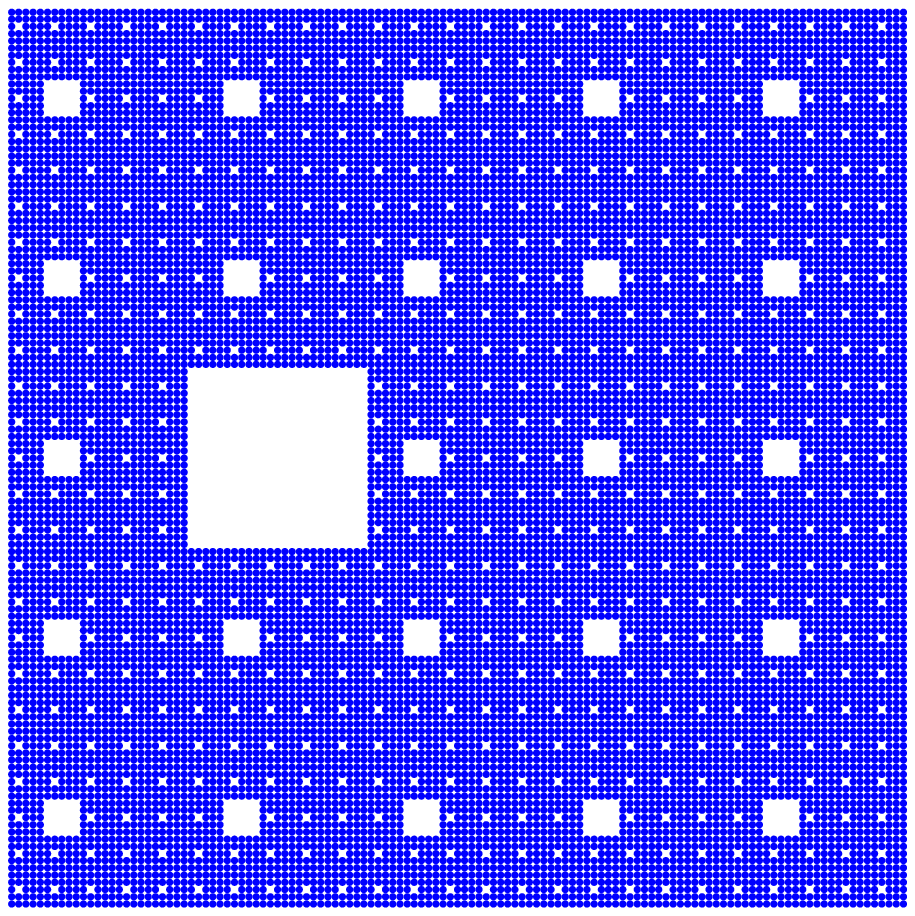}
  }\\
\caption{It is conjectured \cite{RuanWang} that the above two fractal squares are not Lipschitz equivalent.}
\label{fig:n=5}
\end{figure}

%In the remain of this section, we briefly recall some notions to be used
%later on.

In the following, we give a precise description of our result.
Let $(M,d)$ be a locally compact, separable metric space, let $\mu$ be a Radon measure on $M$ with full support. We call $(M,d,\mu)$ a \emph{metric measure space.} Set
\begin{equation*}
V(x,r):=\mu (B(x,r))
\end{equation*}%
to be the volume of the ball $B(x,r)$ .

Let $(M,d,\mu)$ be a metric measure space.
 Let $C(M)$ be the continuous function space. For any $1\leq p\leq\infty$, let $L^p(M,\mu)$ be the Lebesgue function space, and set
\begin{equation*}
||u||_p=||u||_{L^p(M,\mu)}.
\end{equation*}
For any $\sigma>0$, define the functional $W_{\sigma,M}(u)$ on measurable functions on $M$ by
\begin{equation}
W_{\sigma,M}(u):=\sup\limits_{0<r<1}r^{-2\sigma}\int_M\left[\frac{1}{V(x,r)}\int_{B(x,r)}|u(x)-u(y)|^2d\mu(y)\right]d\mu(x).
\end{equation}
Define the space $W^{\sigma,2}$ as follows,
\begin{equation*}
W^{\sigma,2}=W^{\sigma,2}(M,d,\mu):=\{u\in L^2:W_{\sigma,M}(u)<\infty\}.
\end{equation*}
Thus $W^{\sigma,2}$ is a Banach space with the norm
\begin{equation*}
||u||_{\sigma,2}:=||u||_2+W_\sigma(u)^{1/2},
\end{equation*}
and it is one of the family of \emph{Besov spaces}, see \cite[Section 4]{GRIHL03}. Set
\begin{equation}
\beta^\ast:=2\sup\{\sigma: W^{\sigma,2} \text{ contains non-constant functions }\}.
\end{equation}
Following \cite{GRIHL03}, we call $\beta^\ast$ the \emph{critical exponent} of the family $W^{\sigma,2}$ of the Besov spaces in $(M,d,\mu)$. Note that the critical exponent $\beta^\ast$ is uniquely determined by $(M,d,\mu)$.

Let $F$ be a closed set in $\mathbb{R}^n$. A Borel measure $\mu$  on $F$
is said to be \emph{Alfors regular}, if there exists $0<\alpha\leq n$ and a constant $C>0$ such that for any Euclidean metric ball $B(x,r)$ with $0<r<1$,
\begin{equation}
C^{-1}r^\alpha\leq \mu(B(x,r))\leq Cr^\alpha.
\end{equation}
%where $\mu$ is some Borel measure on $F$.
 It is known(\cite{JONW84}) that such $\mu$ is equivalent to the $\alpha-$dimensional Hausdorff measure on $F$.
 The set $F$ is called an $\alpha-$set if it admits an Alfors regular measure.
 %while in our setting we replace $d$ by $\alpha$ since our $d$ means metric.
 From now on, we shall always use $d$ to denote the Euclidean metric on ${\mathbb R}^n$, and use $\mu$ to denote
 the $\alpha$-dimensional Hausdorff measure. (Actually, when we write $(F,d,\mu)$, we mean that $\mu$ is the restriction of the $\alpha$-dimensional Hausdorff measure on $F$.)

\begin{theorem}\label{T1}
Let $F$ be an $\alpha-$set in $\mathbb{R}^n$. Let $T: F\to T(F)$ be a bi-Lipschitz transform. Denote by $\beta_1^\ast$ and $\beta_2^\ast$ the critical exponents of $(F,d,\mu)$ and $(T(F), d, \mu)$, respectively. Then we have
\begin{equation}
\beta_1^\ast=\beta_2^\ast.
\end{equation}
\end{theorem}

\begin{example}\label{example:gasket}{\rm
Let $\left\{K_1,\{F_i\}_{i=1}^3\right\}$ be the Sierpi\'{n}ski gasket, $\left\{K_2,\{G_j\}_{j=1}^{27}\right\}$ be a generalized gasket\textbf{(rotation free)}, where the IFS are showing by Fig. \ref{F1} and Fig. \ref{F2}, respectively.

\begin{figure}[th]
\begin{tabular}{cc}
\begin{minipage}[t]{2.5in} 
\includegraphics[width=2.5in]{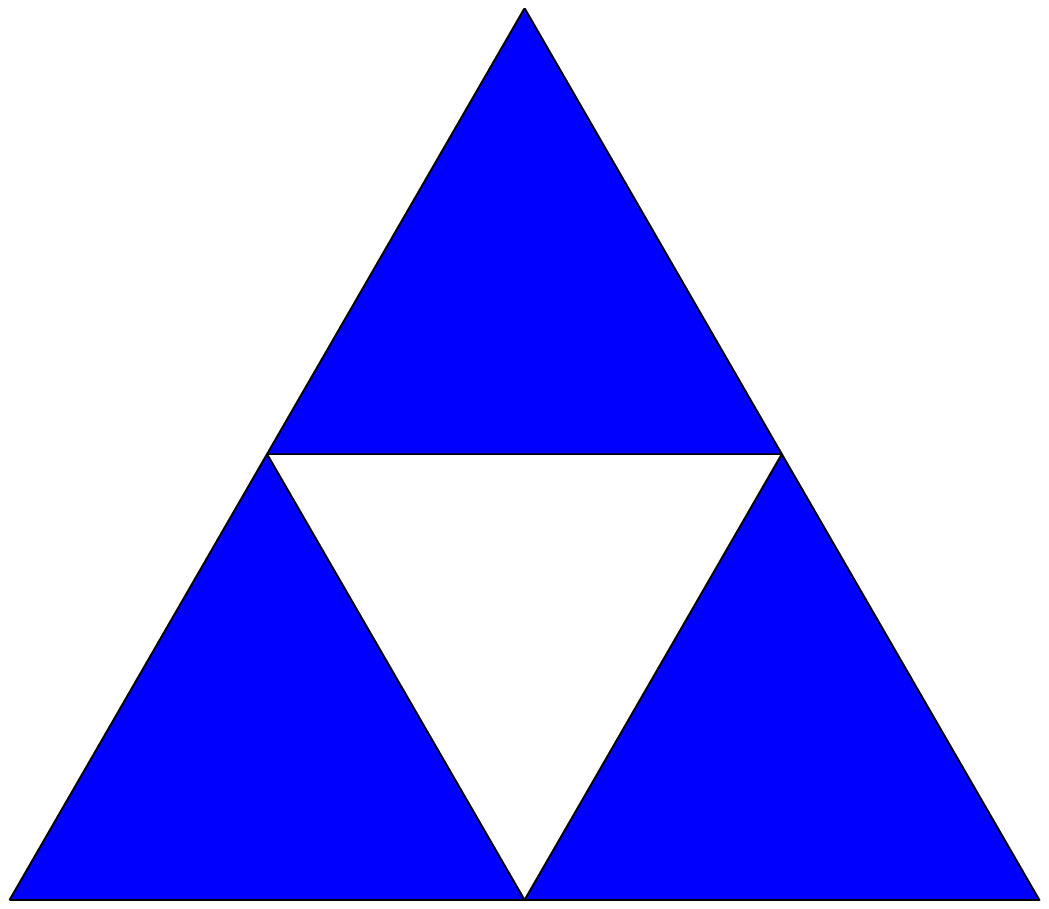}
\caption{IFS of $K_1$}\label{F1} \end{minipage} 
\begin{minipage}[t]{2.5in}
\includegraphics[width=2.5in]{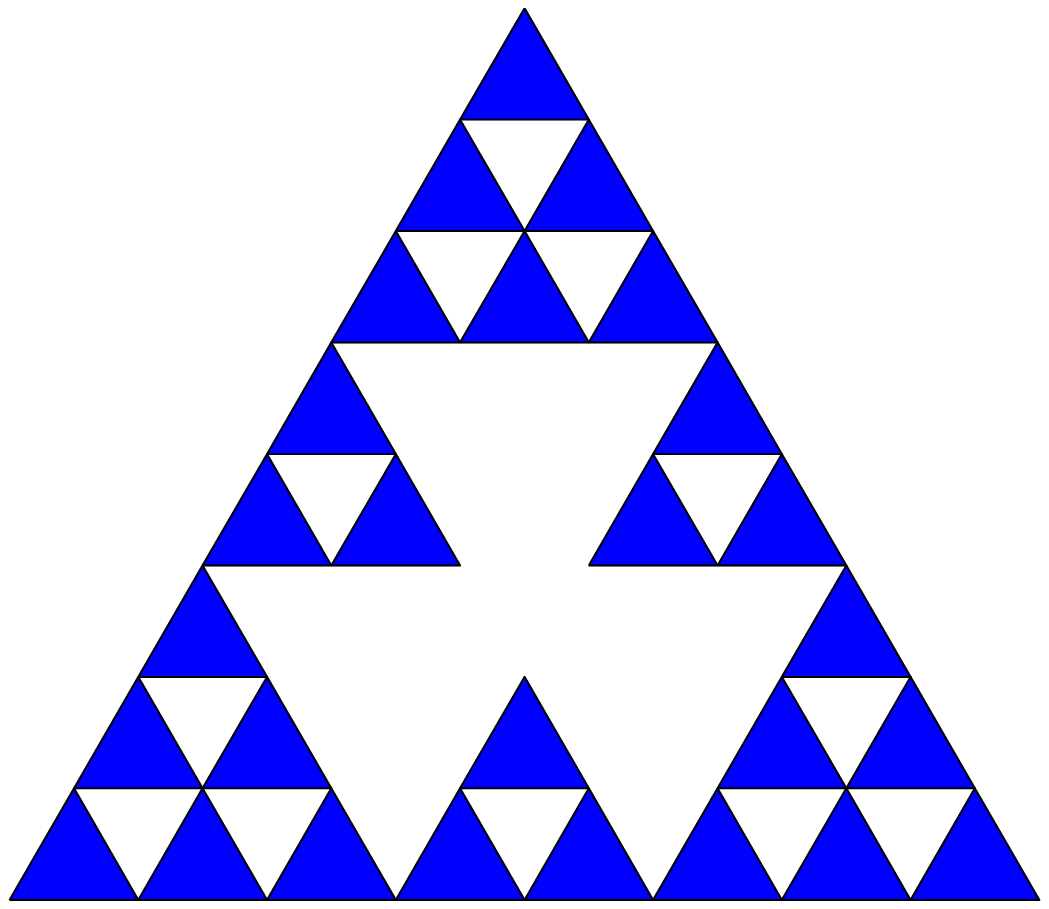} \caption{IFS of $K_2$}\label{F2}
\end{minipage} &
\end{tabular}
\end{figure}

We use the IFS of $K_1$ and $K_2$ in the above to construct some other self-similar sets.
Let $K_3$ and $K_4$ be the self-similar set  generated by the IFS
$$\{G_i\circ F_j\}_{1\leq i\leq 3, 1\leq j\leq 27} \quad \text{ and }\quad
\{F_j\circ G_i\}_{1\leq j\leq 27, 1\leq i\leq 3},$$
respectively.
See Fig. \ref{F3} and Fig. \ref{F4}.
It is clear that the Hausdorff dimensions of these sets are all equal to  $\alpha= \log3/\log2 $.

\begin{figure}[th]
\begin{tabular}{cc}
\begin{minipage}[t]{2.5in} 
\includegraphics[width=2.5in]{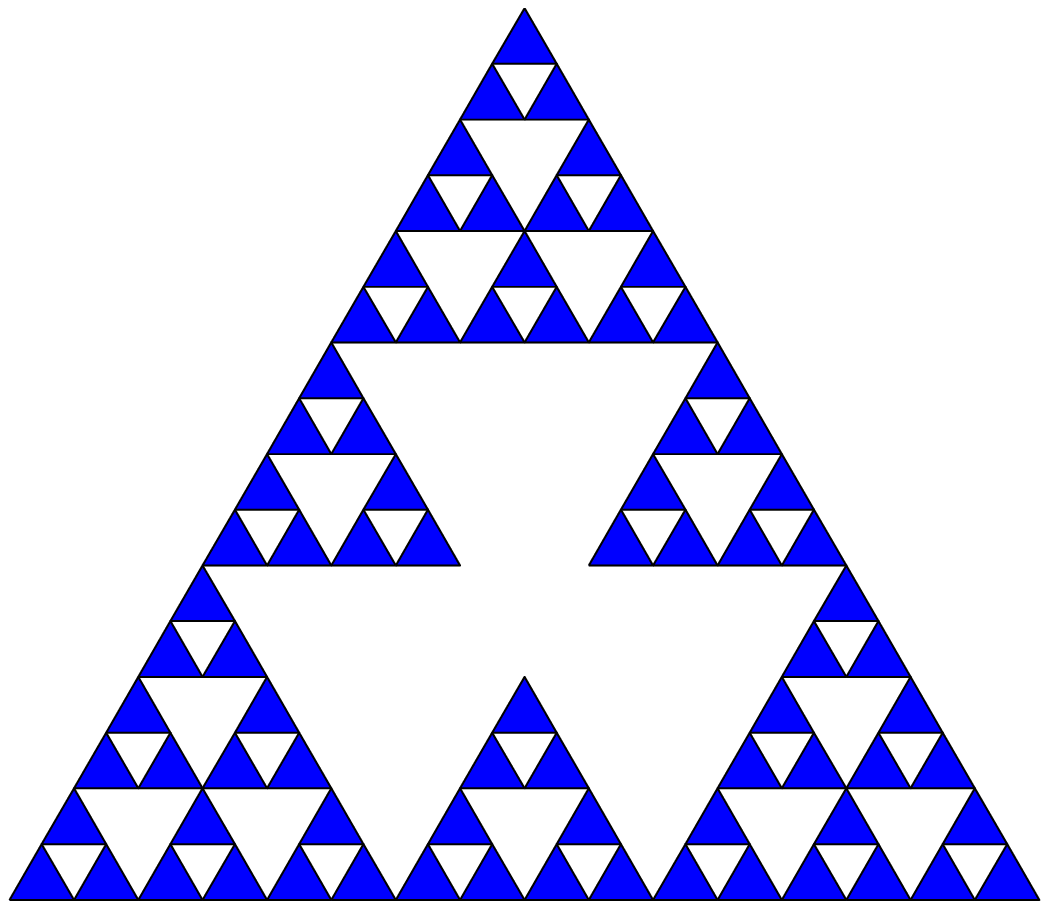}
\caption{IFS of $K_3$}\label{F3} \end{minipage} 
\begin{minipage}[t]{2.5in}
\includegraphics[width=2.5in]{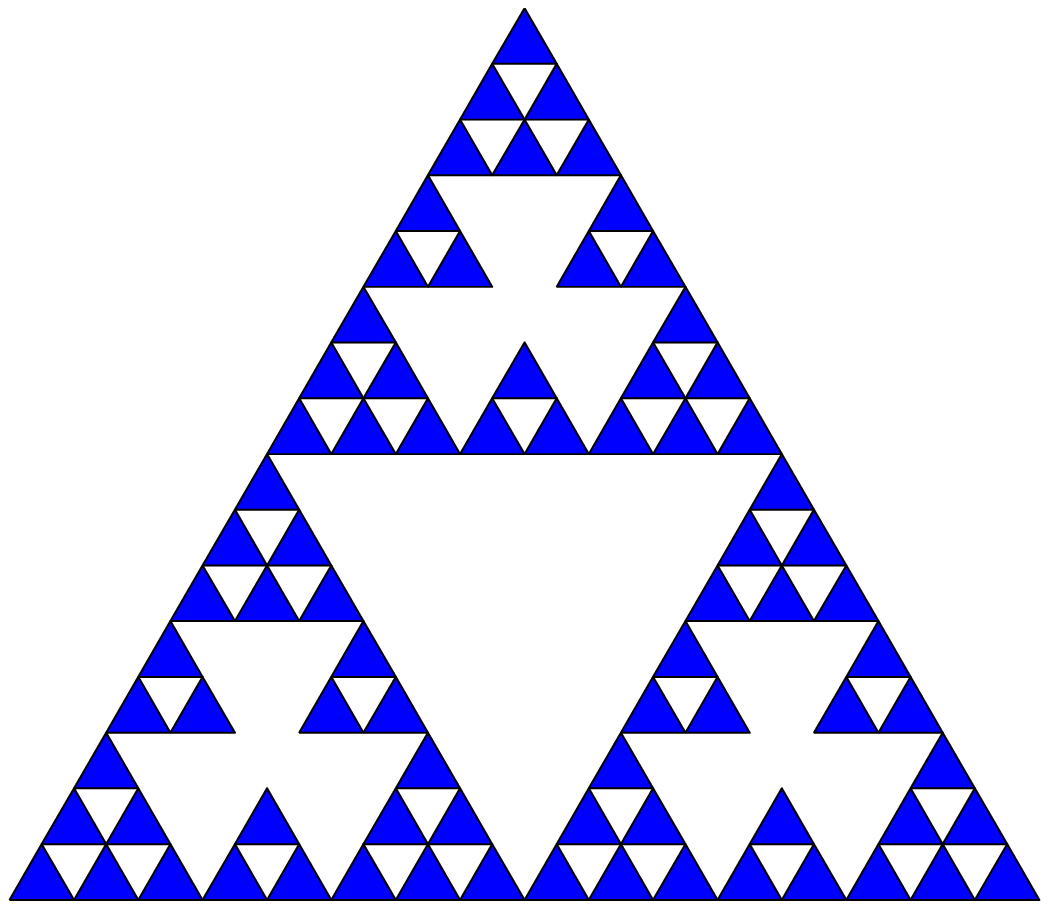} \caption{IFS of $K_4$}\label{F4}
\end{minipage} &
\end{tabular}
\end{figure}

 Using Theorem \ref{T1},
 we shall show that any two of $K_1$, $K_2$, $K_3$ and $K_4$ are not Lipschitz equivalent, except the pair $K_3$ and $K_4$.

 Actually $K_3$ and $K_4$  are not homeomorphic by a connectedness argument. Indeed,
 we can remove six points on $K_3$ to obtain six connected components, while it is impossible on $K_4$.
}
\end{example}

The paper is organized as follows. In Section 2, we proof Theorem \ref{T1}. In Section 3, we discuss how to calculate $\beta^*$ by using Dirichlet \textbf{forms}.

\section{\textbf{Proof of Theorem \ref{T1}}}

First, we give two lemmas. Recall that\textbf{ $\mu$ denotes} the $\alpha$-dimensional Hausdorff measure.

\begin{lemma}\label{L1}
Assume that all the assumptions in Theorem \ref{T1}  hold. Then for any $1\leq p\leq\infty$, there exists a constant $\lambda_p$ only depending on $F,T$ and $p$ such that for any measurable function $v$ on $F$,
\begin{equation}\label{pnorm}
||v\circ T^{-1}||_{L^p(T(F))}\leq \lambda_p||v||_{L^p(F)}.
\end{equation}
\end{lemma}
\begin{proof} It is well known that for any Borel set $A\subseteq F$,
\begin{equation}
C^{-1}\mu(A)\leq\mu(T(A))\leq C\mu(A).
\end{equation}
(See for example, \cite{Fal90}.) Therefore, an approximation argument by using simple functions leads to (\ref{pnorm}).
\end{proof}

\begin{lemma}\label{L2}
Assume that all the assumptions in Theorem \ref{T1}  hold. Then for any $\sigma>0$, there exist constants $C,C'>0$ only depending on $F,T$ and $\sigma$ such that for any $0<r<1$ and any measurable function $u$ on $F$,
\begin{align}
&\int_{T(F)}\left[\int_{B(x,r)}|u\circ T^{-1}(x)-u\circ T^{-1}(y)|^2d\mu(y)\right]d\mu(x)\notag\\
\leq &C'\int_{F}\left[\int_{B(x,Cr)}|u (x)-u (y)|^2d\mu(y)\right]d\mu(x).\label{Besovnorm}
\end{align}
\end{lemma}
\begin{proof}
Fix $x\in T(F)$, then by using the fact that $T$ is Lipschitz, there exists $C>0$ which is independent of $x$ such that
\begin{equation*}
T^{-1}(B(x,r))\subseteq B(T^{-1}x,Cr).
\end{equation*}
Applying Lemma \ref{L1} \textbf{with $p=2$ to}
\begin{equation*}
v\circ T^{-1}(y)=|u\circ T^{-1}(x)-u\circ T^{-1}(y)|\cdot1_{B(x,r)}(y),
\end{equation*}
we obtain
\begin{align}
&\int_{B(x,r)}|u\circ T^{-1}(x)-u\circ T^{-1}(y)|^2d\mu(y)\notag\\
\leq&\lambda_2\int_{T^{-1}(B(x,r))}|u\circ T^{-1}(x)-u(z)|^2d\mu(z)\notag\\
\leq& \lambda_2\int_{B(T^{-1}x,Cr)}|u\circ T^{-1}(x)-u(z)|^2d\mu(z).
\end{align}
Thus by integrating on $x$ over $T(F)$, we get
\begin{align}
&\int_{T(F)}\left[\int_{B(x,r)}|u\circ T^{-1}(x)-u\circ T^{-1}(y)|^2d\mu(y)\right]d\mu(x)\notag\\
\leq& \lambda_2\int_{T(F)}\left[\int_{B(T^{-1}x,Cr)}|u\circ T^{-1}(x)-u(z)|^2d\mu(z)\right]d\mu(x).\label{e1}
\end{align}
Applying Lemma \ref{L1} again\textbf{ with $p=1$ to}
\begin{equation*}
 v\circ T^{-1}(x)=\int_{B(T^{-1}(x),Cr)}|u\circ T^{-1}(x)-u(z)|^2d\mu(z),
\end{equation*}
we have
\begin{align}
&\int_{T(F)}\left[\int_{B(T^{-1}x,Cr)}|u\circ T^{-1}(x)-u(z)|^2d\mu(z)\right]d\mu(x)\notag\\
\leq& \lambda_1\int_F\left[\int_{B(x,Cr)}|u(x)-u(y)|^2d\mu(y)\right]d\mu(x).\label{e2}
\end{align}
Combining (\ref{e1}) and (\ref{e2}), we have (\ref{Besovnorm}) with $C'=\lambda_1\lambda_2$.
\end{proof}

\begin{proof}[\textit{Proof of Theorem \ref{T1}.}]
We only need to show that $\beta_1^\ast\leq\beta_2^\ast$. For any $\sigma<\beta_1^\ast$, we can find a non-constant function $u\in W^{\sigma,2}(F,d,\mu)$, thus $||u||_{L^2(F)}<\infty$ and $W_{\sigma,F}(u)<\infty$.

Firstly, by Lemma \ref{L1}, we have $u\circ T^{-1}\in L^2(T(F))$.
Secondly, by Lemma \ref{L2}, we have
\begin{align}
W_{\sigma,T(F)}(u\circ T^{-1})\leq&\sup\limits_{0<r<1}C'r^{-2\sigma-\alpha}\int_{T(F)}\left[\int_{B(x,r)}|u\circ T^{-1}(x)-u\circ T^{-1}(y)|^2d\mu(y)\right]d\mu(x)\notag\\
\leq &C'\sup\limits_{0<r<1}r^{-2\sigma-\alpha}\int_{F}\left[\int_{B(x,Cr)}|u (x)-u (y)|^2d\mu(y)\right]d\mu(x)\notag\\
\leq&C'\sup\limits_{0<r<C}r^{-2\sigma-\alpha}\int_{F}\left[\int_{B(x,r)}|u (x)-u (y)|^2d\mu(y)\right]d\mu(x)\notag\\
\leq&C'\left\{\sup\limits_{0<r<1}+\sup\limits_{1\leq r<C}\right\}r^{-2\sigma-\alpha}\int_{F}\left[\int_{B(x,r)}|u (x)-u (y)|^2d\mu(y)\right]d\mu(x)\notag\\
\leq&C'W_{\sigma,F}(u)+2C'\int_{F}|u(x)|^2\sup\limits_{x\in F}V(x,C)\notag\\
\leq&C'W_{\sigma,F}(u)+C'C^\alpha||u||_{L^2(F)}<\infty.\notag
\end{align}
Therefore $u\circ T^{-1}$ is a non-constant function in $W^{\sigma,2}(T(F),d,\mu)$, which implies that $\sigma<\beta_2^\ast$. Since $\sigma<\beta_1^\ast$ is arbitrary, we conclude that $\beta_1^\ast\leq\beta_2^\ast$.
\end{proof}

\section{\textbf{Computation of walk dimension}}\label{sec:example}

In this section, we are concerned with how to calculate $\beta^\ast$ of some self-similar sets.
 %and use our Theorem to prove that they are not Lipschitz equivalent if they have different $\beta^\ast$'s.

\begin{notation}
The sign $f\asymp g$ means that there exists constant $c>0$ such that $
c^{-1}f\leq g\leq cf$. The letters $C,C_{i},C^{\prime },C_{i}^{\prime }$
etc. denote constants whose values are not important and may change from
line to line.
\end{notation}

The following result shows that under a mild condition, $\beta^*$ coincides with the walk dimension
and hence, we can use the techniques in analysis on fractals to calculate $\beta^*$.

 \begin{theorem}\label{T2} (\cite[Theorem 4.6]{GRIHL03}) If there exists a heat kernel $p_t$ on $(M,d,\mu)$ satisfying a sufficient decay condition: ($0<t<t_0$, $\Phi$ is some nonnegative decreasing function on $[0,+\infty)$)
\begin{align}
&p_t(x,y)\asymp \frac{C}{t^{\alpha/\beta}}\Phi\left(c\frac{d(x,y)}{t^{1/\beta}}\right). \label{eq:1}\\
&\int_0^\infty s^{\alpha+\beta+\varepsilon}\Phi(s)\frac{ds}{s}<\infty. \label{eq:2}
\end{align}
Then $\beta=\beta^\ast$. (The number  $\beta$ satisfying \eqref{eq:1} and \eqref{eq:2} is called the walk dimension of $M$.)
\end{theorem}

In the following, we use the self-similar sets in Example \ref{example:gasket} to illustrate how to calculate the walk dimension.

\smallskip

 \subsection{Dirichlet form} By using $(\Delta-Y)$-transforms (e.g.\cite{KIG01} or \cite{STR06}), we can construct standard Dirichlet forms.

On $\left\{K_1,\{F_i\}_{i=1}^3\right\}$, define
\begin{align*}
  \mathcal{E}^{(1)}(u) & :=\lim\limits_{m\rightarrow\infty}\left(\frac{5}{3}\right)^{m}\sum\limits_{x\underset{m}{\sim} y}(u(x)-u(y))^2 \\
  \mathcal{F}^{(1)} & :=\{u\in C(K_1) : \mathcal{E}^{(1)}(u)<\infty\}.
\end{align*}
On $\left\{K_2,\{G_j\}_{i=1}^{27}\right\}$, define
\begin{align*}
  \mathcal{E}^{(2)}(u) & :=\lim\limits_{m\rightarrow\infty}\left(\frac{295}{63}\right)^{m}\sum\limits_{x\underset{m}{\sim} y}(u(x)-u(y))^2 \\
  \mathcal{F}^{(2)} & :=\{u\in C(K_2) : \mathcal{E}^{(2)}(u)<\infty\}.
\end{align*}
We denote the resistance scaling constants in these Dirichlet forms by
$$r_1^{-1}=\frac{5}{3}, \ r_2^{-1}=\frac{295}{63},$$
and we set
$$\gamma_1=\frac{\log\left({5}/{3}\right)}{\log2}, \quad \gamma_2=\frac{\log\left({295}/{63}\right)}{\log8}.$$

Similarly, for $K_3$ and $K_4$, the $(\Delta-Y)$-transform gives us
$$r_3^{-1}=r_4^{-1}=r_1^{-1}\cdot r_2^{-1}=\frac{1475}{189},$$
 thus for this two sets, the standard Dirichlet forms have the same expression. Let $$\gamma_3=\gamma_4=\frac{\log\left({1475}/{189}\right)}{\log16}.$$

\subsection{Heat kernel estimates}
Using standard ways (e.g.\cite{BAR98},\cite{KUM93}),for instance, first deducing Nash-type inequality to obtain the existence and on-diagonal upper bound of the heat kernel, and together with the estimation of the exist time, one can obtain the off-diagonal upper bound; Using the upper bound to get a near-diagonal lower bound, and together with a chain argument( in all of our examples, the Euclidean metric satisfies the \emph{chain condition}\cite{GRIHL03}), one can obtain the off-diagonal lower bound. We list the two-sided estimates of heat kernels associated with these Dirichlet forms as follows.
Let $p^i_t(x,y)$ be the heat kernel of $(\mathcal{E}^{(i)},\mathcal{F}^{(i)})$ on $L^2(K_i,\mu)$, $i=1,2,3,4$. Then we have
\begin{equation}\label{hk}
p^i_t(x,y)\asymp \frac{C}{t^{\alpha/\beta_i}}\exp\left(-c\left(\frac{d(x,y)}{t^{1/\beta_i}}\right)^{\beta_i/(\beta_i-1)}\right)
\end{equation}
where $\beta_i=\alpha+\gamma_i$.

These heat kernels are with exponential decay and satisfy the conditions in \cite[Theorem 4.6]{GRIHL03}, thus $\beta_i=\beta_i^\ast$ by Theorem \ref{T2}. Hence, by Theorem \ref{T1},
 we see that any two of $K_1$, $K_2$, $K_3$ and $K_4$ are not Lipschitz equivalent, except the pair $K_3$ and $K_4$.

\bibliographystyle{siam}
%\bibliography{ref}

\end{document}

%% file: tcirm.tex
\def\oint{\fint}

\def\FRAME#1#2#3#4#5#6#7#8
{
 \begin{figure}[ptbh]
 \begin{center}
 \includegraphics[height=#3]{#7}
 \caption{#5}
 \label{#6}
 \end{center}
 \end{figure}
}

%% file: walk-dim_2016.bbl
\begin{thebibliography}{1}

\bibitem{BAR98}
{\sc M.~Barlow}, {\em Diffusions on fractals}, vol.~{\bf 1690} of Lect. Notes
  Math., Springer, 1998, pp.~1--121.

  \bibitem{DS}
G.~David and S.~Semmes, {\it Fractured fractals and broken dreams :
self-similar geometry through metric and measure}, Oxford Univ.\
Press,~1997.

  \bibitem{FaMa92} K. J.~Falconer and D. T.~Marsh, {\it On the Lipschitz equivalence of
Cantor sets}, Mathematika, \textbf{39} (1992), 223--233.

%\bibitem{FRZ14} A. H. Fan, H. Rao and Y. Zhang,
%\emph{Higher dimensional Frobenius problem: Maximal saturated cone,
%growth function and rigidity,}

\bibitem{Fal90} K. J. Falconer,
{\it Fractal Geometry, Mathematical Foundations and Applications,} Wiley, New York, 1990.


\bibitem{GRIHL03}
{\sc A.~Grigor'yan, J.~Hu, and K.-S. Lau}, {\em Heat kernels on metric-measure
  spaces and an application to semilinear elliptic equations}, Trans. Amer.
  Math. Soc., \textbf{355} (2003), pp.~2065--2095.

\bibitem{JONW84}
{\sc A.~Jonsson and H.~Wallin}, {\em Function spaces on subsets of $R^n$},
  Math. Reports Vol. {\bf 2}, Acad. Publ., Harwood, 1984.

\bibitem{KIG01}
{\sc J.~Kigami}, {\em Analysis on Fractals}, Cambridge Univ. Press, 2001.

\bibitem{KUM93}
{\sc T.~Kumagai}, {\em Estimates of the transition densities for Brownian motion on nested fractals}, Probab. Theory and Related Fields, \textbf{96} (1993), pp.~205--224.

\bibitem{LuoLau13} J. J. Luo and K. S. Lau,
{\it Lipschitz equivalence of self-similar sets and hyperbolic boundaries},
Adv. Math., \textbf{235} (2013), 555--579.

\bibitem{RRW12}
H. Rao, H. J. Ruan, and Y. Wang,  {\it Lipschitz equivalence of Cantor sets and algebraic properties of contraction ratios},  Trans. Amer. Math. Soc.,  \textbf{364} (2012), 1109-1126.

\bibitem{RRW13} H. Rao, H. J. Ruan and Y. Wang, \emph{Lipschitz equivalence of self-similar sets: algebraic and geometric properites,} Contemp. Math., \textbf{600} (2013).

\bibitem{RRX06} H.~Rao, H. J.~Ruan and L. F.~Xi, {\it Lipschitz equivalence of self-similar
sets}, C. R. Acad. Math. Sci. Paris., \textbf{342} (2006),
191--196.

\bibitem{RZ15} H. Rao and Y. Zhang,
\emph{Higher dimensional Frobenius problem: Maximal saturated cone,
growth function and rigidity,} \emph{J. Math. Pures Appl.}, \textbf{104} (2015), 868-881.

\bibitem{RaoZhu} H. Rao and Y. J. Zhu, {\it Lipschitz equivalence of fractal squares which are not totally disconnected},
Preprint 2016.

\bibitem{RuanWang} H. J. Ruan and Y. Wang, {\it Topological invariants and Lipschitz equivalence of fractal squares},
Preprint 2015.

\bibitem{RWX12} H. J. Ruan, Y. Wang, L. F. Xi,
{\it Lipschitz equivalence of self-similar sets with touching structures},
 Nonlinearity,  \textbf{27} (2014), 1299--1321.

\bibitem{STR06}
{\sc R.~Strichartz}, {\em Differential equations on fractals: a tutorial},
  Princeton University Press, 2006.

  \bibitem{Why} G. T. Whyburn,  {\it Topological characterization of the Sierpinski curve,}
   Fund. Math., \textbf{45} (1958), 320--324.

\bibitem{XiXi10} L. F. Xi and Y. Xiong, {\it Self-similar sets with initial cubic patterns},
 C. R. Math. Acad. Sci. Paris, \textbf{348} (2010),  15-20.

\bibitem{XiXi12}L.  F. Xi and  Y. Xiong, {\it Lipschitz equivalence of fractals generated by nested cubes},
 Math Z.,  \textbf{271} (2012), 1287--1308.

%\bibitem{XiXi13} L. F. Xi and Y. Xiong,  \emph{Lipschitz equivalence class, ideal class and the Gauss class number problem,}
%Preprint 2013 (arXiv:1304.0103 [math.MG]).
\end{thebibliography}
